\documentclass[a4paper,12pt]{amsart}
\usepackage{amssymb}
\usepackage[utf8]{inputenc}
\usepackage{amsmath}
\usepackage{stmaryrd}
\usepackage{amscd,amsthm,amssymb}
\usepackage{enumerate}
\usepackage{color}
\usepackage[all,cmtip]{xy}

\scrollmode
\usepackage{latexsym}

\def\.{\cdot}
\def\a{\alpha}
\def\b{\beta}
\def\c{\gamma}

\def\la{\langle}
\def\ra{\rangle}

\def\s{\sigma}

\def\beq{\begin{equation}}
\def\eeq{\end{equation}}
\def\bea{\begin{eqnarray*}}
\def\eea{\end{eqnarray*}}
\def\beaa{\begin{eqnarray}}
\def\eeaa{\end{eqnarray}}
\def\ba{\begin{array}}
\def\ea{\end{array}}

\def\o{\omega}

\def \CM{\mathbb{C}}

\def\M{M}

\def\Ric{\mathrm{Ric}}

\def\be{\begin{equation}}
\def\ee{\end{equation}}
\def\tr{\mathrm{tr}}

\def\Sym{\mathrm{Sym}}

\def\gg{\mathfrak{g}}

\def\k{\mathfrak{k}}

\def\mm{\mathfrak{m}}

\def\u{\mathfrak{u}}

\def\U{\mathrm{U}}

\def\E{\mathcal{E}}

\def\H{\mathcal{H}}

\def\SO{\mathrm{SO}}

\def\End{\mathrm{End}}

\def\Sp{\mathrm{Sp}}

\def\Sym{\mathrm{Sym}}

\def\Id{\mathrm{id}}

\def\T{\mathrm{\,T}}

\def\pr{\mathrm{pr}}

\def\s{\mathrm{scal}}

\newtheorem{epr}{Proposition}[section]
\newtheorem{ath}[epr]{Theorem}
\newtheorem{elem}[epr]{Lemma}
\newtheorem{ecor}[epr]{Corollary}

\theoremstyle{definition}
\newtheorem{ede}[epr]{Definition}
\newtheorem{ere}[epr]{Remark}

\title[Quaternion-K\"ahler manifolds with non-negative quaternionic curvature]{Quaternion-K\"ahler manifolds with non-negative quaternionic sectional curvature}

\author{Andrei Moroianu, Uwe Semmelmann, Gregor Weingart}

\address{Andrei Moroianu \\ Université Paris-Saclay, CNRS,  Laboratoire de mathématiques d'Orsay, 91405 Orsay, France, and Institute of Mathematics ``Simion Stoilow'' of the Romanian Academy, 21 Calea Grivitei, 010702 Bucharest, Romania }
\email{andrei.moroianu@math.cnrs.fr}

\address{Uwe Semmelmann, Institut f\"ur Geometrie und Topologie, Fachbereich Mathematik, Universit{\"a}t Stuttgart, Pfaffenwaldring 57, 70569 Stuttgart, Germany
}
\email{uwe.semmelmann@mathematik.uni-stuttgart.de}

\address{Gregor Weingart\\
 Instituto de Matemáticas\\
 Universidad Nacional Autónoma de México\\
 Avenida Universidad s/n\\
 Colonia Lomas de Chamilpa\\
 62210 Cuernavaca, Morelos, Mexico}
\email{gw@matcuer.unam.mx}

\date{\today}

\begin{document}

\begin{abstract} Compact Hermitian symmetric spaces are Kähler manifolds with constant scalar curvature and non-negative sectional curvature. 
A famous result by A. Gray states that, conversely, a compact simply connected Kähler manifold with constant scalar curvature and non-negative sectional curvature is a Hermitian symmetric space. The aim of the present article is to transpose Gray's result to the quaternion-Kähler setting. In order to achieve this, we introduce the quaternionic sectional curvature of quaternion-Kähler manifolds, we show that every Wolf space has non-negative quaternionic sectional curvature, and we prove that, conversely, every quaternion-Kähler manifold with non-negative quaternionic sectional curvature is a Wolf space. The proof makes crucial use of the nearly Kähler twistor spaces of positive quaternion-Kähler manifolds.
\end{abstract}

\maketitle

\section{Introduction}

The original inspiration for this article stems from a beautiful result of A. Gray \cite{G} stating that {\em compact simply connected Kähler manifolds with constant scalar curvature and non-negative sectional curvature are Hermitian symmetric}. Gray's proof was revisited in \cite{MSW}, where it is shown that it basically follows from a Weitzenböck formula applied to the curvature tensor $R$, together with the fact that the curvature operator $q(R)$ is semi-definite on symmetric tensors whenever the sectional curvature is non-negative, cf. \cite{HMS}. 

It is thus natural to ask whether this strategy can be applied to other classes of Riemannian manifolds, like quaternion-Kähler manifolds, which appear in Berger's classification of Riemannian manifolds with special holonomy. In every dimension $4n\ge 8$, the holonomy group of a quaternion-Kähler manifold is contained in $\Sp(n)\cdot\Sp(1)\subset\SO(4n)$. Geometrically, a quaternion-Kähler manifold $(M,g)$ can be characterized by the existence of a parallel rank 3 vector subbundle $\E$ of $\End(\T M)$ locally spanned by three almost Hermitian structures $I,J,K$ satisfying the quaternionic relations $I^2=J^2=K^2=-\Id$, and $IJ=K$. 

All quaternion-Kähler manifolds are Einstein. The scalar curvature vanishes if and only if the manifold is hyperkähler (i.e. $I,J,K$ can be chosen globally as Kähler structures), a situation which is in general not considered as proper quaternion-Kähler. In the remaining cases, a quaternion-Kähler manifold $(M,g,\E)$ is called positive or negative according to the sign of its scalar curvature. While in the negative case it is possible to construct complete non-locally symmetric examples of quaternion-Kähler manifolds, the only examples of complete positive quaternion-Kähler manifolds are the so-called Wolf spaces \cite{W}, which are symmetric spaces of compact type associated to each simple compact Lie group $G$ via the choice of a root of $G$ of maximal length.

The famous LeBrun-Salamon conjecture (which was proved only in small dimensions $n=2$ \cite{PS} and $n=3$ or $4$ \cite{BW}) states that all positive quaternion-Kähler manifolds are Wolf spaces. A possible weakening of this conjecture would be to add some curvature positivity assumption, and to use Gray's theorem \cite{G}.
Indeed, an important feature of positive quaternion-Kähler manifolds is that the sphere bundle of $\E$, called the {\em twistor space}, can be naturally endowed with a Kähler-Einstein structure \cite{S}. 

It is thus tempting to try to impose a condition on the sectional curvature of $M$ which is sufficient for the non-negativity of the sectional curvature of the twistor space and conclude by Gray's theorem. However this cannot work, since the twistor spaces of Wolf spaces endowed with the Kähler-Einstein metric are not symmetric, except for the quaternionic projective space $\mathbb{H} P^n$. Such a curvature condition on $M$ would therefore be too strong and would not hold for the other Wolf spaces. 

Our strategy in this article is inspired by this naive idea. However, the main new ingredient is that, instead of using the Kähler-Einstein structure of the twistor space, we rather consider its {\em nearly Kähler structure}, obtained from the former by changing the sign of the complex structure on the vertical distribution, and rescaling the metric by a factor 1/2 in vertical directions. Recall that every nearly Kähler manifold $(N,g,J)$ carries a metric connection $\bar\nabla$ with parallel skew-symmetric torsion, called the canonical connection. This connection  often is more appropriate for understanding the intrinsic structure of the manifold than the Levi-Civita connection $\nabla^g$, and indeed, it will play a key role in what follows.

Note that a somewhat similar approach was tempted by Chow and Yang in \cite{CY}. They introduced a so called quaternionic bisectional curvature and showed that if it is non-negative for a positive quaternion-K\"ahler manifold, the manifold has to be a Wolf space. In fact, without noticing, they proved that the manifold has to be $\mathbb H P^n$. By calling the non-negativity of the  quaternionic bisectional curvature "a somewhat weaker assumption" than non-negativity of the sectional curvature, they suggest that the first is implied by the later. Actually, this is not the case, as explained in \cite{MSW}.

Let us now give a quick overview of the organization of the paper. 
In Section 2 we focus on nearly Kähler manifolds and generalize Gray's result to this setting. In Proposition \ref{nkcurv} we show that the curvature tensor $\bar R$ of the canonical connection is the sum of a curvature tensor of Kähler type and a parallel tensor, which, according to Proposition \ref{wbf-curv}, satisfies a Weitzenböck formula similar to the one in the Kähler setting. Assuming the non-negativity of the sectional curvature of $\bar R$, it is then possible to show that $\bar\nabla\bar R=0$, so the manifold is an Ambrose-Singer homogeneous space (cf. Theorem \ref{nk-hom}).

In Section 3 we study in detail the nearly Kähler structure of the twistor spaces of positive quaternion-Kähler manifolds. Most calculations performed in \S3.1 can be found at different places in the literature, but the computations in the references that we could find are not always correct, so we give a self-contained proof of the necessary formulae in Proposition \ref{nk}. We then relate in \S3.2 the curvature tensor of the canonical connection $\bar \nabla$ on the nearly Kähler twistor space of a positive quaternion-Kähler manifold, to the Riemannian curvature of the quaternion-Kähler base (cf. Proposition \ref{sec-twist}).

In Section 4 we introduce the notion of quaternionic sectional curvature for quaternion-Kähler manifolds, and obtain our main result (Theorem \ref{theorem2}) which can be stated as follows: {\em A compact quaternion-Kähler manifold with non-negative quaternionic sectional curvature is a Wolf space}. The main idea is to reinterpret the non-negativity condition for the quaternionic sectional curvature of the base in terms of  the non-negativity of the sectional curvature of the canonical connection on the nearly Kähler twistor space, and then to use Theorem \ref{nk-hom}.

Finally, in \S4.2 we show that, conversely,  every Wolf space has non-negative quaternionic sectional curvature (cf. Corollary \ref{cor-sec}), and in Proposition \ref{estimates} we obtain sharp estimates for the sectional curvature of 2-planes $P$ on every Wolf space in terms of the scalar curvature and the so-called Wirtinger angle $\theta_P$.

 {\sc Acknowledgments.} A.M. was partly supported by the PNRR-III-C9-2023-I8 grant CF 149/31.07.2023 {\em Conformal Aspects of Geometry and Dynamics}. U.S. was partly supported by the Procope Project No. 57650868 (Germany) /  48959TL (France).

%
\section{Nearly K\"ahler manifolds of non-negative curvature}
%

A {\it nearly K\"ahler manifold}  \ is by definition a Riemannian manifold $(N,g)$ together with an almost complex structure $J$, orthogonal
with respect to $g$ and satisfying the condition $(\nabla^g_X J)X = 0$ for all tangent vectors $X$. Here $\nabla^g$ is the Levi-Civita connection of $g$.
A nearly K\"ahler manifold is called {\it strict} if $\nabla^g_XJ\ne0$ for every non-zero tangent vector $X$.
On any nearly K\"ahler manifold there exists a unique metric connection $\bar\nabla$ with skew-symmetric and parallel torsion, preserving the
Hermitian structure, i.e.  with $\bar\nabla J=0$. It can be written
as $\bar\nabla = \nabla^g + \tau$ with $\tau := - \frac12 \, J \circ \nabla^g J$.  Then $\tau(X,Y,Z) = g(\tau_XY, Z)$ is a 
$\bar\nabla$-parallel $3$-form. Note that
$\tau_X \in \Lambda^{(2,0)+(0,2)}\T N$, i.e. $\tau_X(JY, JZ) = - \tau_X(Y, Z)$ holds for all vectors $X, Y,Z$. 
We will call $\bar \nabla$ the {\it canonical connection} of the nearly K\"ahler manifold $(N, g, J)$.


For the convenience of the reader the  following proposition recalls and collects properties of the curvature of the canonical connection on a nearly K\"ahler manifold. 
These facts will be  important in the rest of our paper.

\begin{epr}\label{nkcurv}
Let $(N^{2n}, g, J)$ be a  strict nearly K\"ahler manifold. Then the curvature $\bar R$ of the canonical connection $\bar \nabla$  is
pair symmetric and satisfies the first and second Bianchi identities 
$$
\mathop{\mathfrak{S}}_{XYZ}  \big( \bar R(X,Y,Z,W) - 4g (\tau_XY, \tau_ZW) \big) \;=\; 0\qquad \mbox{and} \qquad \mathop{\mathfrak{S}}_{XYZ} ( \bar\nabla_X \bar R)_{Y,Z} \;=\; 0 \ , 
$$
for all tangent vectors $X,Y,Z,W \in \T\ (=\T N)$.
The Ricci tensor  $\overline{\Ric}$ of $\bar R$ is  symmetric and $\bar \nabla$-parallel. Moreover, the curvature   $\bar R$ takes values in the Lie algebra 
$\u(n) \cong \Lambda^{1,1} \T$, i.e.  $\bar R \in \Sym^2  \Lambda^{1,1} \T$, and  $\bar R$ can be
written as
$$
\bar R \, = \, R^{K} +  R^{0} \ ,
$$
where $R^{K}$ is a K\"ahler curvature tensor and $R^{0}$ is a $\bar \nabla$-parallel tensor. More precisely, it holds that 
$R^{0}_{X,Y} =- (\sigma_{X,Y} + \sigma_{JX, JY})$,
where $\sigma := \tau_{e_i} \wedge \tau_{e_i} = \frac12 d \tau$ (throughout this paper we use the Einstein sum convention).
\end{epr}
\proof
The curvature of any metric connection $\bar \nabla = \nabla^g + \tau $ with parallel skew-symmetric torsion $\tau$ can be decomposed as 
\be\label{rb}\bar R = R^g - \tau^2\ ,
\ee
where $\tau^2_{X,Y}: = [\tau_X, \tau_Y] - 2 \tau_{\tau_XY}$ (see \cite[Eq. (3)]{CMS}). This formula implies the pair symmetry of $\bar R$ and
thus also the symmetry of the Ricci tensor $\overline{\Ric}$ of $\bar R$. The first and second Bianchi identities follow from \cite[Thm. 5.3, Ch. III]{KN}, (cf. also \cite[Cor. 2.3]{CMS}).
The fact that  the Riemannian Ricci tensor ${\Ric}^g$ is  $\bar \nabla$-parallel was first shown by P.-A. Nagy \cite[Cor. 2.1]{N}. By \eqref{rb}, it follows that $\overline{\Ric}$ is  $\bar \nabla$-parallel as well. Note that in the case of interest for us in this paper, namely when the nearly 
K\"ahler structure is the twistor space over a positive quaternion-Kähler manifold, we provide a direct proof for this fact in Proposition \ref{p-ric} below.

Since the almost complex structure is $\bar\nabla$-parallel, the curvature $\bar R$ takes values in $\u(n)$ and the holonomy of $\bar\nabla$ is a subgroup of $\U(n)$. The pair symmetry implies $\bar R \in \Sym^2 \Lambda^{1,1} \T$. 

It remains to prove the decomposition of $\bar R$. We have $\bar R \in \Sym^2 (\Lambda^2 \T) \cong \ker b \oplus \Lambda^4 \T$,
where $\mathrm{b} : \Sym^2 (\Lambda^2 \T) \rightarrow \Lambda^4 \T$ is the Bianchi map, defined by 
$\mathrm{b}(R):= e_i \wedge e_j \wedge R(e_i\wedge e_j)$. The kernel of $\mathrm{b}$ is by definition the 
space of algebraic Riemannian curvature tensors. Consider the embedding $\imath$ of $\Lambda^4 \T$
into $ \Sym^2 (\Lambda^2 \T)$ defined by $\imath(\sigma)(X\wedge Y) := \sigma(X,Y,\cdot, \cdot)$. Since $\mathrm{b}\circ\imath=12\,\mathrm{Id}_{\Lambda^4 T}$, it follows  that $ \frac{1}{12} \imath\circ \mathrm{b}$ is the projection
onto the subspace of $ \Sym^2  (\Lambda^2 \T)$ isomorphic to  $\Lambda^4 \T$. An easy computation shows that 
$\mathrm{b}(\bar R) = -\mathrm{b}(\tau^2) = 8 \sigma$. Hence, one can write $\bar R_{X,Y} =:  \tilde R_{X,Y}  + \frac{1}{12} (\imath \circ \mathrm{b} (\bar R))_{X, Y}  = \tilde R_{X,Y} + \frac23 \sigma_{X,Y}$,
with  
\be\label{sig}\sigma_{X,Y} = 2 \tau(e_i, X, Y) \tau_{e_i} - 2  \tau_{e_i}X \wedge  \tau_{e_i}Y =: 2\sigma^-_{X,Y}  - 2\sigma^+_{X,Y}\ ,\ee
where $\sigma^-$ denotes the
first and $\sigma^+$ the second summand of $\sigma$ in the above expression and $\tilde R$ is an algebraic Riemannian curvature tensor. 
Since $\tau_X \in \Lambda^{(2,0) + (0,2)} \T$, we see that $\sigma^+_{X,Y} \in  \Lambda^{1,1} \T$, i.e. $\sigma^+  \in \Sym^2 \Lambda^{1,1} \T$
and $\sigma^-_{X,Y} \in  \Lambda^{(2,0) + (0,2)} \T$ for all $X,Y\in\T$. In particular, we have $\sigma^+_{X,Y} = \frac12(\sigma_{X,Y} + \sigma_{JX,JY})$. 
Another simple calculation shows $\mathrm{b}(\sigma^-) = 2 \sigma$ and $\mathrm{b}(\sigma^+) =-4 \sigma$.
It follows that $\tilde \sigma := 2\sigma^- + \sigma^+$ is in the kernel of the Bianchi map, i.e. $\tilde\sigma$ is an algebraic Riemannian curvature tensor, whence  $R^{K}:= \tilde R + \frac23 \tilde\sigma$ is also an algebraic Riemannian curvature tensor.
Moreover, we can write $\sigma = 2\sigma^--2\sigma^+ = \tilde\sigma - 3\sigma^+$ and 
$$
\bar R_{X,Y} = \tilde R_{X,Y} + \frac23 \sigma_{X,Y} = (\tilde R_{X,Y}+ \frac23 \tilde\sigma_{X,Y}) - 2\sigma^+_{X,Y} = R^{K}_{X,Y} 
-(\sigma_{X,Y} + \sigma_{JX,JY}) \ .
$$
Since $\bar R$ and $\sigma^+$ are in $\Sym^2\Lambda^{1,1} $ the previous calculation shows that $R^{K}$ is an algebraic
Riemannian curvature tensor of K\"ahler type. From \eqref{sig} it follows that $\sigma$ is  $\bar\nabla$-parallel and thus 
the same is true for $\sigma^+$. Hence $R^{0}$ defined by $R^{0}_{X, Y} := -(\sigma_{X,Y} + \sigma_{JX,JY})$ is a  $\bar \nabla$-parallel tensor.
\qed



Let  $q(\bar R) $ be the curvature endomorphism of the vector bundle $\Lambda^2 \T  \otimes \Lambda^2 \T $, defined as
$$
q(\bar R) K \, :=\,   \frac12 (e_i \wedge e_j)_* \bar R_{e_i, e_j} K\ ,
$$
for every section $K$ of $\Lambda^2 \T  \otimes \Lambda^2 \T $, where $\{e_i\}$ is any local orthonormal basis of $\T$. If $(\o_\a)_{1\le\a\le n(n-1)/2}$ is an orthonormal basis of $\Lambda^2 \T$, then one can write $q(\bar R) K \, =\,  (\o_\a)_* \bar R(\o_\a)_* K$.

\begin{ere}
The curvature endomorphism $q(\bar R)$ is symmetric. This follows from the pair symmetry of $\bar R$ which more generally  holds for the curvature tensor
of all connections with parallel skew-symmetric torsion. In particular, $q(\bar R)$ is pointwise diagonalizable. 
\end{ere}

Consider the curvature tensor $\bar R$ as an element of  $\Omega^2(\Lambda^2 \T )$, i.e. as a 2-form with values in the vector bundle $\Lambda^2 \T $. 
\begin{elem}\label{l2.2}
Let $\pi_{\Sym^2}$ denote the projection from $\Lambda^2 \T  \otimes \Lambda^2 \T $ to $\Sym^2(\Lambda^2 \T)$. Then 
$$\pi_{\Sym^2}(e_j \wedge e_i \, \lrcorner \, \bar R_{e_i, e_j} \bar R)= \frac12 q(\bar R) \bar R\ ,$$
where the interior and exterior products on the left hand side are only applied to the first factor of $\Lambda^2 \T  \otimes \Lambda^2 \T $.
\end{elem}

\begin{proof}
Let us write $\bar R=\bar R_{\a\b}\o_\a\otimes\o_\b$, with $\bar R_{\a\b}=\bar R_{\b\a}$ for every $1\le\a,\b\le n(n-1)/2$. We then get
\bea
q(\bar R) \bar R&=&\bar R_{\a\b}(\o_\a)_*(\o_\b)_*\bar R=\bar R_{\a\b}\bar R_{\c\delta}(\o_\a)_*(\o_\b)_*(\o_\c\otimes\o_\delta)\\
&=&\bar R_{\a\b}\bar R_{\c\delta}\big((\o_\a)_*(\o_\b)_*\o_\c\otimes\o_\delta+(\o_\b)_*\o_\c\otimes(\o_\a)_*\o_\delta\\&&\qquad\qquad+(\o_\a)_*\o_\c\otimes(\o_\b)_*\o_\delta+\o_\c\otimes(\o_\a)_*(\o_\b)_*\o_\delta\big)\\
&=&\bar R_{\a\b}\bar R_{\c\delta}\big((\o_\a)_*(\o_\b)_*\o_\c\otimes\o_\delta+2(\o_\a)_*\o_\c\otimes(\o_\b)_*\o_\delta+\o_\c\otimes(\o_\a)_*(\o_\b)_*\o_\delta\big)\ .
\eea
On the other hand, recall that  $(e_i \wedge e_j)_* $ acts on  $\Lambda^2 \T$
as $   \, e_j \wedge e_i \, \lrcorner \,  -\,  e_i \wedge e_j \, \lrcorner $. Therefore, using that $\bar R(\o_\a)=\bar R_{\a\b}\o_\b$, we can write
\bea
e_j \wedge e_i \, \lrcorner \, \bar R_{e_i, e_j} \bar R&=&\frac12(e_j \wedge e_i -e_i\wedge e_j\lrcorner)\,  \bar R_{e_i, e_j} \bar R=(\o_\a)_*\bar R(\o_\a) \bar R\\
&=&(\o_\a)_*\big(\bar R_{\a\b} \bar R_{\c\delta}((\o_\b)_*\o_\c\otimes\o_\delta+\o_\c\otimes(\o_\b)_*\o_\delta)\big)\\
&=&\bar R_{\a\b} \bar R_{\c\delta}\big((\o_\a)_*(\o_\b)_*\o_\c\otimes\o_\delta+(\o_\a)_*\o_\c\otimes(\o_\b)_*\o_\delta\big)\ ,
\eea
where we recall that in the first two lines, $(\omega_\a)_*$ only acts on the first factor of $\Lambda^2 \T  \otimes \Lambda^2 \T $.

Thanks to the symmetry of $\bar R$, the last summand belongs to $\Sym^2(\Lambda^2 \T)$, whereas the projection of the first factor onto $\Sym^2(\Lambda^2 \T)$ reads
\bea\pi_{\Sym^2}(\bar R_{\a\b} \bar R_{\c\delta}(\o_\a)_*(\o_\b)_*\o_\c\otimes\o_\delta)&=&\frac12\bar R_{\a\b} \bar R_{\c\delta}((\o_\a)_*(\o_\b)_*\o_\c\otimes\o_\delta+\o_\delta\otimes (\o_\a)_*(\o_\b)_*\o_\c)\\
&=&\frac12\bar R_{\a\b} \bar R_{\c\delta}((\o_\a)_*(\o_\b)_*\o_\c\otimes\o_\delta+\o_\c\otimes (\o_\a)_*(\o_\b)_*\o_\delta)\ .
\eea
Comparing the above expressions yields the result.
\end{proof}

\begin{epr}\label{wbf-curv}
For the curvature tensor of the canonical connection of a strict nearly K\"ahler manifold  $(N,g,J)$ the following relation holds
$$
\pi_{\Sym^2}(\bar \nabla^* \bar \nabla \bar R) \, = \, - \frac12 q(\bar R) \bar R \ .
$$
\end{epr}
\proof
Considering again $\bar R$ as an element of  $\Omega^2(\Lambda^2 \T )$, it is well known that the second Bianchi identity for $\bar R$ is equivalent 
to $d^{\bar\nabla} \bar R = 0$. Since $\bar\nabla \, \overline{\Ric} = 0$ we also have $\delta^{\bar \nabla} \bar R = 0$. Indeed, writing  $\bar R$ as
$
\bar R = \frac12 \, e_i \wedge e_j \otimes \bar R_{e_i, e_j}
$
we find (choosing the local orthonormal frame to be $\bar\nabla$-parallel at the point where the computation is done):
$$
\delta^{\bar \nabla} \bar R \, =\, - e_i \,  \lrcorner \bar\nabla_{e_i}\bar R\, =\, - \frac12 \, e_i \, \lrcorner (e_j \wedge e_k) \otimes (\bar\nabla_{e_i} \bar R)_{e_j, e_k}
\,= \,- (\bar\nabla_{e_i}  \bar R)_{e_i, e_k}\ .
$$
Using the second Bianchi identity for $\bar R$ we obtain
\bea
  g((\bar\nabla_{e_i} \bar R)_{e_i, e_k}Y, Z)  &=&   g ((\bar\nabla_{e_i} \bar R)_{Y, Z} e_i, X) \\
&=&
  g (  (\bar\nabla_{Y} \bar R)_{e_i, Z} e_i, X)  \,+\,   g ( (\bar\nabla_{Z} \bar R)_{Y, e_i} e_i, X)   \\
 &=&
  Y g(\bar R_{e_i, Z} e_i, X) \,+ \,Z g( \bar R_{Y, e_i} e_i, X) \ .
\eea
The last sum vanishes because of $\bar\nabla \, \overline{\Ric} = 0$. We assume again that the local orthonormal frame is $\bar\nabla$-parallel at the point where the computation is done, so in particular $[e_i, e_j] = -\tau_{e_i} e_j + \tau_{e_j}e_i = -2 \tau_{e_i} e_j $. Keeping in mind that $d^{\bar\nabla} \bar R = 0$ and $\delta^{\bar \nabla} \bar R = 0$, we then obtain:
\bea
0=\delta^{\bar \nabla} d^{\bar\nabla} \bar R 
&=& - e_i \, \lrcorner \, \bar\nabla_{e_i} (e_j \wedge\bar \nabla_{e_j} \bar R)  
= -  e_i \, \lrcorner  \, (e_j \wedge  \bar\nabla_{e_i}\bar \nabla_{e_j} \bar R) \\
&=&  -  \bar \nabla_{e_i} \bar \nabla_{e_i} \bar R  \,  + \,   e_j \wedge e_i \, \lrcorner \,   \bar \nabla_{e_i} \bar \nabla_{e_j}  \bar R \\
&=& \bar \nabla^* \bar \nabla \bar R \, +\,  e_j \wedge e_i \, \lrcorner \,
\big( \bar R_{e_i, e_j} \bar R   +  \bar\nabla_{e_j} \bar\nabla_{e_i}  \bar R + \bar \nabla_{[e_i, e_j] } \bar R  \big)\\
&=&  \bar \nabla^* \bar \nabla \bar R \, -\,  d^{\bar\nabla}  \delta^{\bar \nabla} \bar R \, +\,  e_j \wedge e_i \, \lrcorner \, \bar R_{e_i, e_j} \bar R
\,- 2  e_j \wedge e_i \, \lrcorner \, \bar\nabla_{\tau_{e_i} e_j} \bar R\ ,
\eea
whence 
\be\label{eqr}\bar \nabla^* \bar \nabla R  \, =\, - e_j \wedge e_i \, \lrcorner \, \bar R_{e_i, e_j} \bar R
\,+2  e_j \wedge e_i \, \lrcorner \, \bar\nabla_{\tau_{e_i} e_j} \bar R\ .
\ee
We claim that the last summand in \eqref{eqr} vanishes. Indeed, the second Bianchi identity yields
\bea
 \,2\,  e_j \wedge e_i \, \lrcorner \,  \bar\nabla_{\tau_{e_i} e_j} \bar R
&=&    2   e_j \wedge e_i \, \lrcorner \,  (\bar\nabla_{e_k} \bar R) \, \tau(e_i, e_j, e_k) \\
&=&     \, e_j \wedge e_i \, \lrcorner \,  (e_a \wedge e_b) \otimes  (\bar\nabla_{e_k} \bar R)_{e_a, e_b} \, \tau(e_i, e_j, e_k) \\
&=&    2\, e_j \wedge e_b \otimes  (\bar\nabla_{e_k} \bar R)_{e_a, e_b} \, \tau(e_a, e_j, e_k) \\
&=&     2\, e_j \wedge e_b \otimes  \big( (\bar\nabla_{e_a} \bar R)_{e_k, e_b}  + (\bar\nabla_{e_b} \bar R)_{e_a, e_k}   \big) \, \tau(e_a, e_j, e_k) \\
&=&  \, e_j \wedge e_b \otimes  (\bar\nabla_{e_b} \bar R)_{e_a, e_k} \, \tau(e_a, e_j, e_k) \ .
\eea
(The last equality follows by exchanging the indices $a$ and $k$ in the third line). We then replace $e_a$ and $e_k$ by $Je_a$ and $Je_k$ and obtain that the last line is equal to its opposite (as $\bar R_{e_a, e_k}=\bar R_{Je_a, Je_k}$ and $\tau(e_a, e_j, e_k)=-\tau(Je_a, e_j, Je_k)$) so it has to vanish. Consequently, \eqref{eqr} becomes 
\be\label{eqr1}\bar \nabla^* \bar \nabla \bar R  \, =\, - e_j \wedge e_i \, \lrcorner \, \bar R_{e_i, e_j} \bar R \ .
\ee
The proposition thus follows from Lemma \ref{l2.2}.
\qed

\begin{ere}\label{proj}
It is easy to check, using the symmetry by pairs of $\bar R$, that $\bar \nabla^* \bar \nabla \bar R$ is actually pair symmetric as well. Equation \eqref{eqr1} thus shows that the formula in Lemma \ref{l2.2} is actually valid without projecting the left hand term on $\Sym^2(\Lambda^2\T)$.
\end{ere}


\begin{elem}\label{non-neg}
Let $(N, g, J) $ be a strict nearly K\"ahler manifold. If the sectional curvature of the curvature $\bar R$ of the canonical connection $\bar \nabla$ is non-negative, then
$
q(\bar R) \ge 0
$
on all symmetric tensors, i.e. all eigenvalues of $q(\bar R)$  are non-negative.
\end{elem}
\proof
Since the curvature $\bar R$ is pair symmetric the proof of the  corresponding statement for the Riemannian curvature carries over
to the present situation (see the proof of \cite[Prop. 6.6.]{HMS}).
\qed

\begin{ath}\label{nk-hom}
A  compact, simply connected strict nearly K\"ahler manifold $(N,g,J) $ such that the sectional curvature of $\bar R$ is non-negative,  is a naturally reductive homogeneous space.
\end{ath}
\proof
We write $\bar R = R^{K}+R^{0}$ as in Proposition \ref{nkcurv}. Since $R^{0} $ is $\bar \nabla$-parallel, $q(\bar R)R_0=0$, so from Proposition \ref{wbf-curv} and Remark \ref{proj}
we get
$$
\bar \nabla^* \bar \nabla   R^{K} \,  = \, \bar \nabla^* \bar \nabla   \bar R \,  = \,   - \frac12 q(\bar R) \bar R  \,  = \,   -\frac12 \,  q(\bar R)   R^{K} \ .
$$
Taking the scalar product with $R^K$ yields 
$$g(R^K,\bar \nabla^* \bar \nabla   R^{K})=-\frac12\,g(R^K,  q(\bar R)   R^{K})\ ,$$
so after integration we find
$$ \| \bar \nabla R^{K}\|^2_{L^2} = -\frac12 \, \la q(\bar R) R^{K}, R^{K}\ra_{L^2} \ .$$
On the other hand, as in the proof of \cite[Lem. 2.4]{MSW} we can replace $R^{K}$ by its holomorphic sectional curvature $S$, which is
a symmetric $4$-tensor. Note that the map $R^{K} \mapsto S$ is injective (see \cite[Lem. 2.2]{JW}).  Then Lemma \ref{non-neg} implies that $g(R^K,  q(\bar R)   R^{K})\ge0$ at every point, whence by the above integral formula $\bar \nabla R^{K} = 0$ and finally $\bar \nabla \bar R = 0$. It follows that the canonical connection $\bar \nabla$
of the nearly K\"ahler manifold $(M,g,J) $ has parallel torsion and parallel curvature. Hence $M$ is homogeneous by the
Ambrose-Singer Theorem \cite{AS}. 
\qed

\bigskip

\section{The twistor space of quaternion-K\"ahler manifolds}

\subsection{The nearly K\"ahler structure of the  twistor space}

Let $(M^{4n}, g^M, \E)$ be a quat\-ernion-K\"ahler manifold, where $\E \subset \End (\T M)$ is the rank 3 parallel subbundle of $ \End (\T M)$  locally spanned by
almost complex structures $I, J, K$ compatible with the metric and satisfying the quaternionic relation $I J = K$. Such a triple $\{ I, J, K \}$ will be called a local quaternionic frame, and will sometimes be written as $\{J_\a\}$, $\a=1,2,3$.
We denote by $\nabla^M$ the Levi-Civita
connection of $g^M$, with $R^M_{X, Y} := \nabla^M_X \nabla^M_Y - \nabla^M_Y \nabla^M_X - \nabla^M_{[X, Y]}$ its curvature tensor, and with  $\s$  the scalar curvature of $R^M$.
The metric $g^M$ is then automatically Einstein as already mentioned in the introduction. 


The {\it twistor space} of $M$ is defined as
$\mathcal Z M := \{ I \in \E \,|\, I^2 = - \Id \}$. It is the total space of a $\CM P^1$-fibration $\mathcal Z M  \rightarrow M$. 
The connection on $\E$ induced by the Levi-Civita of $g^M$ defines the decomposition 
$ \T \mathcal Z M = \mathcal H \oplus \mathcal V$ of the tangent bundle of $\mathcal Z M$ into the horizontal and the vertical tangent spaces. The corresponding decomposition of tangent vectors
$\xi \in \T \mathcal Z M$ will be written as $\xi  = \xi^H + \xi^V$. At a point $I \in\mathcal Z M  $, the
vertical tangent space, i.e. the tangent space to the fibres,  can be written as $\mathcal V_I = \{ A^* \in \E \,|\, A I + IA = 0\}$.
Here and below we will write $A^*$ for an endomorphism $A \in \E$ considered as tangent vector, under the canonical identification of
a vector space with its tangent space at a point. 

On $\mathcal Z M$ we consider the family of metrics $g_c = \pi^* g^M + g_c^{\mathcal V}$ ($c>0$), with the vertical part $g_c^{\mathcal V}$ defined as the restriction $g^\E_c|_{\mathcal V}$
of the metric $g^\E_c(A^*, B^*) = - c^2 \tr (AB) $. For every choice of $c$, $\pi: \mathcal Z M \rightarrow M$ is a Riemannian submersion with
totally geodesic fibres. Since $|A^*|^2 = 4nc^2$ for all $A \in \mathcal Z M$, the twistor fibres $\mathcal Z M_p = \pi^{-1}(p)$, are $2$-spheres of radius
$c \sqrt{4n} $ in the Euclidean vector space $(\E_p, g^\E_c)$. In particular the fibres  $\mathcal Z M_p$ have sectional curvature $\frac{1}{4nc^2}$.

For a tangent vector $X \in \T_pM$ we denote with $\tilde X \in \T_I \mathcal Z M$ its horizontal lift at $I \in \mathcal ZM_p$. The
twistor space $\mathcal Z M$ carries two natural almost complex structures $J^\varepsilon$. They are defined on the horizontal part $\mathcal H_I$ by
$J^\varepsilon (\tilde X)_I = \widetilde{IX}$ and on the vertical part  $\mathcal V_I$ by $J^\varepsilon (A^*)_I = \varepsilon (IA)^*$, with
$\varepsilon = \pm 1$.

\begin{elem}\label{com}
For all vertical vector fields  $A^*$ and  vector fields $X, Y$ on $M$ we have
\begin{enumerate}
\item[(i)]  \;  $J^\varepsilon [A^*, \tilde X] = [J^\varepsilon A^*, \tilde X]$;
\medskip
\item[(ii)]\; $   [A^*, J^\varepsilon \tilde X]^{\mathcal H} = \widetilde{A X} $;
\medskip
\item[(iii)]  \;  $ [\tilde X, \tilde Y]_I = \widetilde{[X, Y]} - [R^M_{X,Y}, I]^*$ in a point $ I \in \mathcal Z M $.
\end{enumerate}
\end{elem}
\proof
These formulae can be found in \cite{B}: the argument for $(i)$ is in \S14.72, and $(ii)$ is exactly equation $(14.72)$. The 
horizontal part of $(iii)$ is clear. The vertical part follows from (9.53a) and (9.53b).
\qed

\begin{elem}\label{curv}
Let $\{I, J, K\}$ be a local quaternionic frame. Then the relation
$$
[R^{M}_{X,Y}, I] = \frac{\s}{4n(n+2)} \big(  - g^M(KX, Y) J^* + g^M(JX, Y) K^* \,  \big)
$$
is true for all $X, Y \in \T_pM$. The formula holds for all even permutations of $I, J, K$.
\end{elem}
\proof
The formula appears in \cite[Thm. 14.39]{B}. Note that  in  \cite[Lem. 14.40]{B} the wrong factor $2$ has to be deleted.
\qed
%
%

\medskip

It is well known that the twistor space of a positive quaternion-K\"ahler manifold is K\"ahler, respectively nearly K\"ahler for certain choices of the
parameter $\varepsilon$ and $c$. For the convenience of the reader and since we need the explicit form of the torsion $\tau$, which we
could not found in the literature, we give the following proposition together with its proof.

\begin{epr}\label{nk}
The twistor space $(\mathcal Z M, g_c, J^\varepsilon)$ is K\"ahler if and only if  $\varepsilon = 1$ and $c^2 = \frac{n+2}{\s}$, and 
nearly K\"ahler if and only if $\varepsilon = -1$ and $c^2 = \frac{n+2}{2 \,\s}$.
The torsion $\tau$ of the nearly K\"ahler structure on $\mathcal Z M$  is  a section of $ \Lambda^2 \mathcal H \otimes \mathcal V$.
At each point $I \in \mathcal Z M$, the torsion is given by
$$
\tau(A^*, \tilde X, \tilde Y)_I = \frac14 g^M(A I X, Y)\ ,
$$
for every $X, Y \in \T_pM$ and $A \in \mathcal V_I$.  In particular, the nearly K\"ahler structure of the twistor space is strict.
\end{epr}
\proof
In order to simplify the notation, we will denote throughout the proof the metric $g_c$ by $g$ and its Levi-Civita connection by $\nabla$. We need to compute the covariant derivative $\nabla J^\varepsilon$. According to the splitting $\T \mathcal Z M = \mathcal H \oplus \mathcal V$, we consider several cases. 

First we note that $(\nabla_{A^*} J^\varepsilon) B^* =  \nabla_{A^*} J^\varepsilon B^* - J^\varepsilon  \nabla_{A^*}  B^* = 0$, since the fibres
are totally geodesic, and $J^\varepsilon$ restricted to the fibres is parallel, being the standard complex structure on $S^2 = \CM P$. 

Next we compute using the Koszul formula 
\bea
2g((\nabla_{\tilde X} J^\varepsilon) A^*, B^*) &=&  2g(\nabla_{\tilde X} J^\varepsilon) A^*, B^*)  + 2g(\nabla_{\tilde X} A^*,  J^\varepsilon B^* )\\ [.5ex]
&=&
\tilde X g(J^\varepsilon A^*, B^{*})  + g([\tilde X, J^\varepsilon A^*], B^*) + g([B^*, \tilde X], A^*) \\ [.5ex]
&&
\quad + \; \tilde X g(A^*, J^\varepsilon  B^{*})  + g([\tilde X,  A^*], J^\varepsilon  B^*) + g([J^\varepsilon  B^*, \tilde X], A^*) \ .
\eea
Since $J^\varepsilon$ is skew-symmetric and using Lemma \ref{com} (i) we see that the summands cancel pairwise and we obtain
$g((\nabla_{\tilde X} J^\varepsilon) A^*, B^*)=0$.

\noindent
As a consequence we also have  $g((\nabla_{A^*} J^\varepsilon) \tilde X, B^*) = - g((\nabla_{A^*} J^\varepsilon) B^*, \tilde X) = 0$.

\noindent
We next claim that $g((\nabla_{\tilde X} J^\varepsilon ) \tilde Y, \tilde Z) = 0$ holds  for all vector fields $X, Y, Z$ on $M$. Indeed,  applying again the 
Koszul formula we find
\bea
2g((\nabla_{\tilde X} J^\varepsilon ) \tilde Y, \tilde Z)  &=& 
\tilde X g(J^\varepsilon  \tilde Y, \tilde Z) + (J^\varepsilon \tilde Y) g(\tilde Z, \tilde X) - \tilde Z g(\tilde X, J^\varepsilon  \tilde Y)\\[.5ex]
&&
+ \quad g([\tilde X, J^\varepsilon \tilde Y ], \tilde Z) - g([J^\varepsilon \tilde Y, \tilde Z], \tilde X) + g([\tilde Z, \tilde X], J^\varepsilon  \tilde Y) \ .
\eea
We can assume  $X,Y,Z$ and $I \in \mathcal Z M_p$ to be parallel at $p$. Then it is clear the the second and the last summand in
the sum above vanish at $I$. Moreover we compute at $I$:
$$
\tilde X g(J^\varepsilon \tilde Y, \tilde Z) = \left.\frac{d}{dt}\right|_{t=0} g(I_t Y, Z) = 0 \ ,
$$
where $I_t$ is the parallel transport of $I$ along the flow of $X$. Hence the scalar product is constant. It follows that the first and third
summands in the sum above vanish as well. If $\varphi_t$ denotes the flow of $\tilde X$, then $\varphi_t(I) = I_t$ and $(J^\varepsilon \tilde Y)_{I_t} = \widetilde{I_tY}$.
For the remaining two summands we compute
$
[\tilde X, J^\varepsilon \tilde Y ]^{\mathcal H} = [\tilde X, \widetilde{AY}]^{\mathcal H}  \ ,
$
where $A$ is an endomorphism on $M$ with $A_p=I$ and parallel along the integral curve of $X$ through $p$.
Then the commutator vanishes since we have $[X,AY] =0$ at $p$.

The parameter $\varepsilon $ and $c$ will be fixed by the following calculations. For every $I\in \mathcal{Z}M$ and $A^*\in\mathcal{V}_I$ we have
$$
g((\nabla_{\tilde X} J^\varepsilon) A^*, \tilde Y)_I = g(\nabla_{\tilde X} (J^\varepsilon A^*), \tilde Y)_I + g(\nabla_{\tilde X} A^*, J^\varepsilon \tilde Y)_I
= -g(\nabla_{\tilde X} \tilde Y,J^\varepsilon A^*)_I + g(\nabla_{\tilde X} A^*, \widetilde{IY})_I \ .
$$

We will calculate the two summands separately. If we write $A = a J + bK$, then $(IA)^* = aK^* - bJ^*$, so using the Koszul formula, Lemma \ref{com} (3) and Lemma \ref{curv}, together with $|K^*|^2 = |J^*|^2 = 4nc^2$, we obtain
\bea
 -g(\nabla_{\tilde X} \tilde Y,J^\varepsilon A^*)_I  &=& \frac12 \,g([\tilde Y, \tilde X], \varepsilon (IA)^*)_I \\
&=&
\frac{\s}{8n(n+2)} \, g \big( - g^M(KX,Y) J^*  + g^M(JX, Y)K^*, \varepsilon (IA)^*\big)_I \\
&=&
\frac{\s \, \varepsilon \, c^2}{2(n+2)} \, g^M(AX, Y)\ .
\eea

Similarly we compute (by extending $I$ to a local section of $\E$):
\bea
g(\nabla_{\tilde X} A^*, \widetilde{IY})_I &=& \frac12 g([\widetilde{IY},  \tilde X],A^*)_I\\
&=&
\frac{\s}{8n(n+2)} g \big( - g^M(KX,IY) J^*  + g^M(JX, IY)K^*, A^*\big)_I \\
&=&
- \frac{\s  \, c^2}{2(n+2)} g^M(AX, Y)\ .
\eea

Combining these two calculations we get: \, $g((\nabla_{\tilde X} J^\varepsilon) A^*, \tilde Y)_I =  \frac{\s  \, c^2}{2(n+2)} (\varepsilon -1)g^M(AX, Y)$.
For later use we note that the formula for the second summand also proves the equation
\be\label{com1}
g([ \widetilde{IY}, \tilde X], A^*)_I \,=\, - \frac{\s  \, c^2}{n+2} \,g^M(AX, Y) \ .
\ee

Finally we have to compute the expression $$g((\nabla_{A^*} J^\varepsilon) \tilde X, \tilde Y)_I = g(\nabla_{A^*} (J^\varepsilon \tilde X), \tilde Y)_I   +  g(\nabla_{A^*} \tilde X,  \widetilde{IY})_I\ .$$
Since  $[\tilde X, A^*]^H=0$, the second summand follows from the calculation in the previous case: 
$$g(\nabla_{A^*} \tilde X,  \widetilde{IY})_I=g(\nabla_{\tilde X}{A^*},  \widetilde{IY})_I=- \frac{\s  \, c^2}{2(n+2)} g^M(AX, Y)\ .$$ 
For the first summand we again use the Koszul
formula to obtain
\be\label{axy}
g(\nabla_{A^*} J^\varepsilon \tilde X, \tilde Y)_I = \frac12 \big(  A^* g(J^\varepsilon \tilde X, \tilde Y)_I + g([A^*, J^\varepsilon\tilde X], \tilde Y)_I - g([J^\varepsilon \tilde X, \tilde Y], A^*)_I\big) \ .
\ee
We fix a curve $A_t \in \mathcal Z M$ with $A_0 = I $ and $\dot A_0 = A$. Then the first summand in \eqref{axy} is equal to $\frac12 \frac{d}{dt}|_{t=0} g(\widetilde{A_tX}, \tilde Y) = \frac12 g^M(AX, Y) $.
Due to Lemma \ref{com}, the second summand in \eqref{axy} is equal to $\tfrac12 g(\widetilde{AX}, \tilde Y) = \frac12 g^M(AX, Y)$. At last, by tensoriality in $X,Y$ and
by \eqref{com1}, the third  summand in \eqref{axy} can be computed as
$
- \frac12 g([\widetilde{IX}, \tilde Y], A^*)_I = - \frac12 \frac{\s  \, c^2}{n+2} g^M(AX, Y)
$.
Altogether  we obtain
\be\label{mix}
g((\nabla_{A^*} J^\varepsilon) \tilde X, \tilde Y)_I =  \left(1 - \frac{\s  \, c^2}{n+2}  \right) g^M(AX, Y)\ .
\ee
Summarizing we see that $(g, J^\varepsilon)$ is K\"ahler, i.e. $J^\varepsilon$ parallel, if and only if  $\s \, c^2 = n+2$ and $\varepsilon = 1$. The twistor space 
 $(g, J^\varepsilon)$ is nearly K\"ahler, i.e. $(\nabla_{\tilde X} J^\varepsilon) A^* +  (\nabla_{A^*} J^\varepsilon) \tilde X = 0$, if and only if
 $\varepsilon = -1$ and $2 \,\s \, c^2 = n+2$.
 
 We still have to compute the torsion $3$-form $\tau$ of the nearly K\"ahler structure on $\mathcal Z M$. From the calculation of
 $\nabla J^\varepsilon$ given above it is clear that $\tau \in \Lambda^2 \mathcal H \otimes \mathcal V$. The explicit form of $\tau$ then follows from
 \eqref{mix} with $2 \, \s \, c^2 = n+2$ and $\varepsilon = -1$. We find
 $$
 \tau(A^*, \tilde X, \tilde Y)_I \, = \, \frac12 \, g((\nabla_{A^*} J^\varepsilon) \tilde X, J^{\varepsilon} \tilde Y)_I \,=\, \frac14 \, g^M(AX, IY) \,=\, \frac14 \, g^M(AIX, Y) \ .
 $$
 Here we use that $A$, representing the tangent vector $A^*$ in $\mathcal V_I$, anti-commutes with $I$.
 For the last statement we still have to show that the torsion has no kernel. Let $A^* + \tilde X$ be an arbitrary tangent vector to the twistor space. Then
 $$
 \tau ( A^* + \tilde X, \tilde X, \tilde Y) \, = \, \tau ( A^* , \tilde X, \tilde Y) \, = \,  \frac14 \, g^M(AIX, Y) \ .
 $$
 Clearly the term $g^M(AIX, Y)$ cannot be zero for all vectors $Y$. Hence the kernel of the torsion $\tau$ has to be trivial and thus the nearly K\"ahler structure 
 on the twistor space is strict.
\qed

\begin{ere}
Our values of $\varepsilon$  and $c$ in the K\"ahler respectively nearly K\"ahler case confirm the results in \cite{AGI}. However, the definition of the constant $c$ in \cite{AGI} has to be modified since it is not consistent with the rest of the calculations.
\end{ere}

From now on we will only consider the nearly K\"ahler structure on the twistor space $\mathcal Z M$, i.e. we fix $\varepsilon = - 1$ and 
$c^2 = \frac{n+2}{2\, \s}$. For this choice we will write $J=J^\varepsilon$, $g=g_c$, and denote as before by $\bar\nabla=\nabla^g+\tau$ the canonical connection of the nearly Kähler manifold $(\mathcal{Z}M,g,J)$.

\medskip

\begin{elem}\label{torsion}
For any vector fields $X, Y$ on $M$ and any vertical vector $V$ in $\mathcal V$ we have
\begin{enumerate}
\item[(i)] \; $ \bar \nabla_{\tilde X} \tilde Y = \widetilde{\nabla^M_X Y}$;
\medskip
\item[(ii)] \; $[\tilde X, \tilde Y]^{\mathcal V} = - 2 \tau_{\tilde X} \tilde Y$;
\medskip
\item[(iii)] \; $g(\bar \nabla_V \tilde X, \tilde Y) = 2 \tau(V, \tilde X, \tilde Y) = 2g (\tau_{\tilde X}\tilde Y, V)$.
\end{enumerate}
In particular, the torsion $\tau$ coincides (up to a sign) with the  O'Neill tensor $A$.
\end{elem}
\proof
The three statements are consequences of the fact that the connection $\bar \nabla$ preserves the splitting $\T \mathcal Z M = \mathcal H \oplus \mathcal V$
in horizontal and vertical vectors. 
For $(i)$ we compute
$$
 \bar \nabla_{\tilde X} \tilde Y \, = \, ( \bar \nabla_{\tilde X} \tilde Y )^{\mathcal H} \,=\, ( \nabla_{\tilde X} \tilde Y + \tau_{\tilde X } \tilde Y) ^{\mathcal H} \,=\, 
 ( \nabla_{\tilde X} \tilde Y) ^{\mathcal H} \,=\,  \widetilde{\nabla^M_X Y} \ ,
$$
where we also use  the property that the torsion  vanishes on three horizontal vectors. Similarly, $(ii)$ follows from:
$$
[\tilde X, \tilde Y]^{\mathcal V} \, = \, (\nabla_{\tilde X} \tilde Y - \nabla_{\tilde Y} \tilde X)^{\mathcal V}  \, = \, 
 (\bar\nabla_{\tilde X} \tilde Y - \tau_{\tilde X } \tilde Y - \bar \nabla_{\tilde Y} \tilde X + \tau_{\tilde Y} \tilde X)^{\mathcal V} \, =\,  - 2 \tau_{\tilde X } \tilde Y  \ .
$$
To prove $(iii)$ it is enough to compute the scalar product with a horizontal lift $\tilde Z$. The statement  then follows from:
$$
g(\bar \nabla_{\tilde X} \tilde Y, \tilde Z) \,=\, g(\nabla_{\tilde X} \tilde Y, \tilde Z) \,=\, g^M(\nabla^M_X Y, Z) \,=\,  g(\widetilde{\nabla^M_X Y}, \tilde Z) \ ,
$$
where for the first equality we used again the fact that the torsion vanishes on three horizontal vectors.
\qed


%
\subsection{The curvature of the nearly K\"ahler twistor space}
%

In this section we will describe the relation between the curvature $\bar R$ of the canonical  nearly K\"ahler connection of the  twistor 
space $\mathcal Z M$ and the  curvature  $R^M$ of the Levi-Civita connection on the quaternion-K\"ahler manifold $M$. Since the decomposition 
$\T \mathcal Z M = \mathcal V \oplus \mathcal H$ is preserved by $\bar\nabla$, it follows that $\bar R$ can be considered as map
$\bar R : \Lambda^2  \T \mathcal Z M \rightarrow \Lambda^2 \mathcal V \oplus \Lambda^2 \mathcal H$. 
Hence, the curvature  $\bar R$ is determined by the following three lemmas.

\begin{elem}\label{hor-curv}
Let $\{I, J, K\}$ be a local quaternionic frame defined in the neighbourhood of some $p\in M$. Then at $I_p \in \mathcal Z M$ we have\\
$$
\bar R(\tilde X, \tilde Y, \tilde Z, \tilde W) =  R^{\M}(X,Y,Z,W) \,-\, \frac{\s}{8n(n+2)} \left( g^{\M}(JX, Y) g^{\M}(JW, Z) + g^{\M}(KX, Y) g^{\M}(KW, Z) \right)
$$
for all $X, Y, Z, W \in \T_p M$.
\end{elem}
\proof
Using ${[\tilde X, \tilde Y]}^{\mathcal H} = \widetilde{[X, Y]}$ and the third formula of Lemma \ref{torsion}  we obtain
$$
\bar R(\tilde X, \tilde Y) \tilde Z  \,=\,  \bar \nabla_{\tilde X} \bar \nabla_{\tilde Y} \tilde Z - \bar \nabla_{\tilde Y} \bar \nabla_{\tilde X} \tilde Z - \bar \nabla_{[\tilde X, \tilde Y]}\tilde Z
\,=\, \widetilde{R^M(X, Y)Z} - \bar \nabla_{[\tilde X, \tilde Y]^{\mathcal V}} \tilde Z\ .
$$
Taking the scalar product with $\tilde W$ in the last summand and applying the first and second  equation of Lemma \ref{torsion}  gives
\bea
-g(\bar \nabla_{[\tilde X, \tilde Y]^V}\tilde Z, \tilde W  ) & =&  2 g(\bar \nabla_{\tau_{\tilde X} \tilde Y} \tilde Z, \tilde W)
 \, =\, 4 g( \tau_{\tilde Z} \tilde W, \tau_{\tilde X} \tilde Y) \\
&=&4 \sum_{i, j=1}^2  \frac{1}{| A^*_i |^2}  \, g(\tau_{\tilde Z} \tilde W, A^*_i) \, g(\tau_{\tilde X} \tilde Y, A^*_j)  \\ 
&=&   \frac{\s}{8n(n+2)}  \left( g^M(KZ, W) g^M(KX, Y) + g^M(JZ, W) g^M(JX, Y)  \right) \ .
 \eea
Here $  \{A_i^*\}$ is an orthogonal basis of $\mathcal V_I$, which we can take to be $\{J^*, K^*\}$, where $\{I, J, K\}$ is  a local quaternionic frame 
with  $|J^*|^2 = |K^*|^2 = 4nc^2 = \frac{2n(n+2)}{\s}$. The last  displayed equation follows from the explicit form of the torsion given in 
Proposition \ref{nk}.
\qed

\bigskip

Consider   $V_1, V_2 \in \mathcal V_I$  two arbitrary vertical tangent vectors. Then, writing  $V_1 = a_1 J^* + b_1 K^*$ and $V_2 = a_2 J^* + b_2 K^*$, the determinant of the family $V_1, V_2$ with respect to every oriented orthonormal basis of $ \mathcal V_I$ is given by 
\be\label{det}\det(V_1, V_2) = (a_1b_2 - a_2 b_1) \frac{2n(n+2)}{\s}\ .
\ee


\medskip

\begin{elem}\label{ver-curv}
For vertical tangent vectors  $V_1, V_2 \in \mathcal V_I$ the following holds\\
$$
\bar R (V_1, V_2, V_2, V_1) =  \frac{\s}{2n(n+2)}   \det(V_1, V_2)^2= \frac{\s}{2n(n+2)}  |V_1\wedge V_2|^2 \ .
$$
\end{elem}
\proof
Since $\tau_{V_1}V_2 = 0$ for vertical vectors $V_1, V_2$ we have  $\bar R (V_1, V_2, V_2, V_1) = R (V_1, V_2, V_2, V_1) $ (here $R$ denotes the Riemannian curvature tensor of $g$ on $\mathcal ZM$). Thus the statement
of the lemma follows from the fact that the fibres of $\mathcal Z M$ are totally geodesic and have sectional curvature 
$\frac{1}{4nc^2} = \frac{\s}{2n(n+2)} $ in the nearly K\"ahler case.
\qed

\bigskip

\begin{elem}\label{mix-curv}
For vertical tangent vectors $V_1,V_2 \in \mathcal V_I$ and tangent vectors $X_1, X_2 \in \T_{\pi(I)}M$ we have
$$
\bar R(\tilde X_1, \tilde X_2, V_2, V_1) = - \frac{\s}{4n(n+2)}\, g^M(IX_2, X_1)   \det(V_1, V_2) \ .
$$
\end{elem}
\proof
We apply the first Bianchi identity for connections with parallel and skew-symmetric torsion (see \cite[Cor. 2.3]{CMS}).
Then two of the curvature terms vanish since $\bar \nabla$ preserves the splitting $\T \mathcal Z M = \mathcal H \oplus \mathcal V$ and
in the cyclic sum over the torsion terms one summand vanishes as well because the torsion is zero when applied to two vertical vectors. It follows 
\bea 
\bar R(\tilde X_1, \tilde X_2, V_2, V_1)  &=& 
4 \big( g( \tau_{\tilde X_2} V_2,  \tau_{\tilde X_1} V_1) \,+\, g(\tau_{V_2} \tilde X_1, \tau_{\tilde X_2} V_1  \big) \\[.5ex]
&=& 
4 \big( -  g(  \tau_{V_1} \,  \tau_{V_2} \tilde X_2, \tilde X_1 ) \, +\,    g( \tilde X_1,  \tau_{V_2} \, \tau_{V_1} \tilde X_2 )  \big)\\[.5ex]
&=& 
-4 \, g( [  \tau_{V_1},  \tau_{V_2}  ]  \tilde X_2, \tilde X_1 )\ .
\eea
It remains to compute the commutator $ [\tau_{V_1},  \tau_{V_2}  ] $. From the explicit form of the torsion, given in Proposition \ref{nk}, we see that
$\tau_{J^*} = \frac14 \tilde K$ and $\tau_{K^*} = - \frac14 \tilde J$, where $\tilde J$ and $\tilde K$ denote the natural lifts of $J$ and $K$ to ${\mathcal H}_I$. 
Hence, writing again $V_1 = a_1 J^* + b_1 K^*$ and $V_2 = a_2 J^* + b_2 K^*$, we obtain
$$
[\tau_{V_1},  \tau_{V_2}  ]  =  \frac18 a_1 b_2 \tilde I  -   \frac18 b_1 a_2  \tilde I = \frac18 (a_1b_2 - b_1a_2) \tilde I  = \frac18 \det(V_1, V_2) \,  \frac{\s}{2n(n+2)} \, \tilde I \ .
$$
The statement now follows from the fact that $g(\tilde I\tilde X_2,\tilde X_1)=g^M(IX_2,X_1)$. 
\qed

\medskip

The explicit curvature formulae above can be used to compute the Ricci curvature of $\bar R$, thus verifying directly that
$\overline{\Ric}$ is $\bar\nabla$-parallel:

\begin{epr}\label{p-ric}
Let $X$ be a horizontal tangent vector of $\mathcal Z M$ and let $V$ be a vertical tangent vector. Then the $\bar\nabla$-Ricci curvature $\overline{\Ric}$ is completely
described  by $\overline{\Ric} (X,V) = 0$ and
$$
\overline{\Ric}(X, X) \, = \,  \frac{(n+1)\,  \s }{ 4n(n+2)} \, |X|_g^2
\qquad \mbox{and} \qquad
\overline{\Ric} (V, V) \, = \,  \frac{ \s }{ 2n(n+2)} \, |V|_g^2 \ .
$$
In particular $\overline{\Ric}$ is  $\bar\nabla$-parallel.
\end{epr}
\proof
For arbitrary tangent vectors $A, B \in \T \mathcal Z M$ the $\bar \nabla$-Ricci curvature  $\overline{\Ric}$ is defined as 
$\overline{\Ric}(A, B)  = \sum_i\bar R(e_i, A, B, e_i)$, where $\{e_i \},  i=1, \ldots, 4n+2,$ is a local orthonormal frame of $\mathcal Z M$,
which can be assumed to be adapted to the decomposition $\T \mathcal Z M = \mathcal V \oplus \mathcal H$. Let now $X\in \mathcal H$ be a horizontal tangent vector and let $V\in \mathcal V$ be a vertical tangent vector. Then, since $\bar \nabla$ preserves
this splitting, we immediately have $\overline{\Ric}(X, V) = 0$. Recall that the Ricci tensor of a connection with parallel skew-symmetric
torsion is symmetric. Hence, it remains to compute $\overline{\Ric}(V, V)$  and $\overline{\Ric}(X, X)$. First, we obtain from Lemma \ref{ver-curv}
\bea
\overline{\Ric}(V, V)  &=&\sum_{i=1}^{4n+2} \bar R (e_i, V, V, e_i) \,=\, \sum_{i=1}^2\frac{\s}{2n(n+2)} |V_i \wedge V |_g^2  \,=\, \frac{\s}{2n(n+2)} |V|_g^2\ ,
\eea
where $\{V_i\}, i=1, 2,$ can be taken to be an orthonormal basis of the vertical tangent space $\mathcal V_I$. Similarly, but this time using a local 
orthonormal basis $\{X_i\}, i=1, \ldots, 4n$,  of the horizontal tangent space $\mathcal H_I$,   the curvature formula of Lemma \ref{hor-curv} leads to
\bea
\overline{\Ric}(X, X)  &=& \sum_{i=1}^{4n+2} \bar R (e_i, X, X, e_i) \,=\, \sum_{i=1}^{4n}\bar R (X_i, X, X, X_i) \\
&=& \Ric^M(\pi_*(X), \pi_*(X))\\
&& \,- \,  \sum_{i=1}^{4n}\frac{\s}{8n(n+2)}  \big(g^M(J\pi_*(X_i), \pi_*(X))^2  + g^M(K\pi_*(X_i), \pi_*(X))^2  \big)\\
&=& \frac{\s}{4n}g^M (\pi_*(X), \pi_*(X))- \frac{\s}{4n(n+2)}g^M (\pi_*(X), \pi_*(X)) \\
&=&\frac{(n+1)\, \s}{4n(n+2)}  g^M (\pi_*(X), \pi_*(X))\ .
\eea
Since $\pi: \mathcal Z M \rightarrow M$ is a Riemannian submersion, we have $g^M (\pi_*(X), \pi_*(X))=g(X, X)$, thus finishing the proof.
\qed

\medskip

Combining the curvature expressions of the preceding three lemmas we obtain a formula for the sectional curvature of $\bar R$. Note that any two tangent vectors in $\T_I\mathcal Z M$ can be written as $\tilde X_1 + V_1, \tilde X_2 + V_2$ for 
$V_1, V_2 \in \mathcal V_I$ and $X_1, X_2 \in \T_{\pi(I)}M$.

\begin{epr}\label{sec-twist} 
Let $V_1,V_2$ be vertical vectors at some $I\in \mathcal Z M$, and let $X_1,\ X_2$ be tangent vectors to $M$ at the point $\pi(I)$. Then
\bea
&& \bar R(\tilde X_1 + V_1, \tilde X_2 + V_2, \tilde X_2 + V_2, \tilde X_1 + V_1) \qquad \\
&& 
\qquad
=\,
R^M(X_1, X_2, X_2, X_1) \, -\,   \frac{\s}{8n(n+2)} \sum_{\alpha = 1}^3 g^M(J_\alpha X_1, X_2)^2\\
&&
\qquad\qquad\qquad\qquad\qquad\qquad\qquad
\, + \; \frac{\s}{8n(n+2)} \left( g^M({IX_2},  X_1)  - 2  \det(V_1, V_2) \right)^2\ .
\eea
\end{epr}
\proof
Using the fact that $\bar \nabla$ preserves the splitting $\T \mathcal Z M = \mathcal H \oplus \mathcal V$, and the pair symmetry of $\bar R$ we obtain
\bea 
\bar R(\tilde X_1 + V_1, \tilde X_2 + V_2, \tilde X_2 + V_2, \tilde X_1 + V_1) &=&
\bar R(\tilde X_1, \tilde X_2, \tilde X_2 , \tilde X_1)  \;+ \; \bar R( V_1, V_2, V_2,  V_1) \\[.5ex]
&&   \qquad   \qquad  \qquad  \qquad  \qquad  +  \; 2 \bar R(\tilde X_1, \tilde X_2 ,  V_2, V_1)  \ .
\eea
If we denote as before the local quaternionic frame $\{I, J, K\}$ by $\{J_\a\}$, $\a=1,2,3$, then substituting
the curvature expressions of Lemmas \ref{hor-curv}, \ref{ver-curv} and \ref{mix-curv} and using the tautological formula 
$$ g^M(JX, Y)^2 + g^M(KX, Y)^2 =\sum_{\alpha = 1}^3 g^M(J_\alpha X, Y)^2- g^M(IX, Y)^2$$
yields the result.
%
%
\qed

\section{Sectional curvature of quaternion-K\"ahler manifolds}

In this section we will derive our main result as a consequence of  Proposition  \ref{sec-twist}.   For doing this we first have to introduce the notions of  quaternionic sectional curvature, quaternionic lines, quaternionic planes and totally real planes.

\subsection{Quaternionic sectional curvature}

Let $(M,g, \E)$ be a quaternion-K\"ahler manifold. Then any tangent vector $X \in \T_pM $ spans a well-defined $4$-dimensional subspace 
$L(X) := \mathrm{span} \{X, IX, JX, KX\} \subset \T_pM$, where $\{I, J, K\}$ is a local quaternionic frame. The space $L(X)$ is called the
{\it quaternionic line} spanned by $X$. We also introduce the $3$-dimensional subspace 
$Q(X) := \mathrm{span} \{IX, JX, KX\} \subset \T_pM$, which is in some sense the imaginary part of $L(X)$.
For non-collinear vectors  $X$ and $Y$, the 2-plane $\mathrm{span}\{X, Y\}$ is called a {\it quaternionic plane} if $Y\in L(X)$ and a {\it totally real plane} if $Y$ is orthogonal to $Q(X)$.

\begin{ede}
Let $(M,g, \E)$ be a quaternion-K\"ahler manifold of dimension $4n\ge 8$. For every $p\in M$ and non-collinear tangent vectors $X,Y\in \T_pM$, the {\em quaternionic sectional curvature} $\kappa_\H(X,Y)$ defined by
\begin{equation}\label{def}
\kappa_\H(X, Y) \, := \,  \kappa(X,Y) \, - \,  \frac{\s}{8n(n+2)} \frac{|X|_g^2\,|\pr_{Q(X)}(Y)|_g^2}{|X\wedge Y|_g^2} \ ,
\end{equation}
where $Q(X)$ is the imaginary part of the quaternionic line generated by $X$ and $\kappa(X,Y)$ is the sectional curvature of the plane spanned by $X, Y$.
\end{ede}

Note that the quaternionic sectional curvature  $\kappa_\H$  coincides with the Riemannian
sectional curvature $\kappa$ on totally real planes, whereas on quaternionic planes, $\kappa_\H$ is equal to $\kappa-\frac{s}{8n(n+2)} $. More  generally, for every plane $P$ and every basis $X,Y$ of $P$, the quantity $\frac{|X|_g^2\,|\pr_{Q(X)}(Y)|_g^2}{|X\wedge Y|_g^2}$ is equal to $\cos^2(\theta_P)$, where $\theta_P$ is the angle between $P$ and the quaternionic line spanned by any of its non-zero vectors, whence
\begin{equation}\label{wirtinger}
\kappa_\H (P) \,  = \,  \kappa(P) \, - \, \frac{\s}{8n(n+2)} \cos^2(\theta_P)\ , 
\end{equation}
for every 2-plane $P$. The angle $\theta_P$ is analogous to the Wirtinger angle in complex geometry.

\bigskip

We can now formulate the main result of our article.

\begin{ath}\label{theorem2}
Let $(M^{4n}, g, \E)$ be a compact positive quaternion-K\"ahler manifold. If the quaternionic sectional curvature
$\kappa_\H$ is non-negative then $M$ is symmetric, i.e. $M$ is isometric to a Wolf space.
\end{ath}
\proof
We start by showing that the quaternionic sectional curvature  $\kappa_\H $ is non-negative if and only if 
on the twistor space $\mathcal Z M$ the sectional curvature of $\bar R$ is non-negative.

Assume that the quaternionic sectional curvature $\kappa_\H$ is non-negative. Then the curvature expression 
in Proposition \ref{sec-twist}, the definition of the  quaternionic sectional curve in \eqref{def} and the formula
$
\sum_{\alpha = 1}^3 g^M(J_\alpha X, Y)^2 = |X|_g^2\,|\pr_{Q(X)}(Y)|_g^2 
$
 show directly that that the  sectional curvature of $\bar R$ is non-negative.
Conversely, if the sectional curvature of $\bar R$ is non-negative, then for any $I\in \mathcal Z M$ and tangent vectors
$X, Y \in \T_{\pi(I)} M$, we can choose vertical vectors $V_1, V_2\in\mathcal V_I$ with $ \det(V_1, V_2)=\frac12g^M({IX_2},  X_1)$. For such 
a choice of $V_1,V_2$, the last summand in the curvature formula of Proposition \ref{sec-twist}  vanishes and the non-negativity of
$\kappa_\H $  follows.

Hence, if the quaternionic K\"ahler sectional curvature $\kappa_\H$  is non-negative, the twistor space $\mathcal Z M$  is homogeneous 
with $\bar \nabla \bar R = 0$  as a consequence of Theorem \ref{nk-hom}.

Moreover, Proposition \ref{sec-twist} also shows that $\bar \nabla \bar R = 0$ if and only if $\nabla^M R^M = 0$. 
In order to see this we first remark that the composition of $\bar R : \Lambda^2  \T \mathcal Z M \rightarrow \Lambda^2 \mathcal V \oplus \Lambda^2 \mathcal H$
with the projection onto the first summand is $\bar\nabla$-parallel. This follows from a general result on the curvature of connections
with parallel skew-symmetric torsion (see \cite[Prop. 3.13]{CMS}). Thus we only need to consider the covariant derivative of the 
purely horizontal part.

From the first step in the proof of Lemma \ref{hor-curv} we obtain the formula
\be\label{curva}
\bar R(\tilde X, \tilde Y, \tilde Z, \tilde V) \,=\, R^M(X,Y,Z,V)\circ \pi \,+\,   4 g( \tau_{\tilde Z} \tilde V, \tau_{\tilde X} \tilde Y) \ .
\ee
Taking the derivative into the direction of a further horizontal vector $\tilde U$ and recalling that the torsion $\tau$ is $\bar\nabla$-parallel 
together with Lemma \ref{torsion}, (3), immediately gives
$$
(\bar\nabla_{\tilde U}\bar R)(\tilde X, \tilde Y, \tilde Z, \tilde V) \,=\, (\nabla_U R^M)(X,Y,Z,V)\circ \pi \ .
$$

This already proves one direction of the statement, i.e. that  $\bar\nabla \bar R = 0$ implies $\nabla^M R^M=0$.

For the other direction we still have to consider the derivative into vertical directions. But then the derivative of the first summand on the right side of \eqref{curva}
vanishes since it is constant along the fibres and the derivative of the second summand vanishes again since the torsion is parallel.

Combining the two statements we see that  $\kappa_\H \ge 0$ implies $\bar \nabla \bar R = 0$, whence $\nabla^M R^M =0$. Thus the quaternion K\"ahler 
manifold $M$ is symmetric, i.e. it  is isometric to a Wolf space.
\qed

Note that using \eqref{wirtinger}, one can restate Theorem \ref{theorem2} by saying that a compact positive quaternion-K\"ahler manifold whose sectional curvature $\kappa$ satisfies for every 2-plane $P$ the inequality
$$\kappa(P)\ge\frac{\s\cos^2(\theta_P)}{8n(n+2)}$$
is a Wolf space. Hence, our main result gives further evidence for the following 

 {\bf Conjecture}: Any compact quaternion-K\"ahler manifold with non-negative sectional curvature is a Wolf space.
 
Recall that by a result of Berger \cite{Ber}, a quaternion-K\"ahler manifold of (strictly) positive sectional curvature is isometric to $\mathbb H P^n$ up to constant rescaling.

\subsection{Sectional curvature of Wolf spaces}
In this subsection we will prove the converse of Theorem \ref{theorem2}, which also further motivates the concept of quaternionic sectional curvature.

\begin{ecor}\label{cor-sec}
The quaternionic sectional curvature of compact Wolf spaces is non-negative. In particular, the sectional curvature on quaternionic planes
is bounded from below by $ \frac{\s}{8n(n+2)}$.
\end{ecor}
\proof
The twistor space $\mathcal Z  M$ of a compact Wolf space  $M$ is a simply connected homogeneous strict nearly K\"ahler manifold and thus
it is a compact naturally reductive $3$-symmetric space $\mathcal Z  M = G/K$ equipped with its canonical complex structure and the adapted
reductive decomposition $\gg = \k \oplus \mm$ (see \cite{Bu}, \cite{N2}). It is known that the canonical connection of the almost Hermitian 
structure of a Riemannian 3-symmetric space coincides with its 
homogeneous connection (see \cite[Prop. 4.1]{GC}). Moreover, the metric $g$ on  $\mathcal Z  M$ is then defined by  the restriction to
$\mm$ of a bi-invariant product on $\gg$ (see \cite[p. 360]{G1} or \cite[Sec. 6]{GC}). Hence we can compute the sectional curvature of the canonical connection  $\bar \nabla$ 
on the twistor space of a Wolf space by using \cite[Thm. 2.6, Ch. X]{KN2}.  At the base point $o$
and for tangent vectors $X, Y \in \mm \cong \T_o G/K$, we find
$$
g(\bar R(X,Y)Y, X) = - g([[X,Y]_\k, Y], X) = g([X,Y]_\k, [X, Y]_\k)  \ge 0 \ .
$$
It follows that the sectional curvature of $\bar R$ is non-negative. Hence, as showed in the proof of Theorem \ref{theorem2}, the quaternionic 
sectional curvature  $\kappa_\H$ is non-negative and the estimate is a consequence of the definition of $\kappa_\H$ in \eqref{def}.
\qed

%

\medskip

Using the Wirtinger angle $\theta_P$ introduced in  \eqref{wirtinger} we can formulate a more precise estimate for the sectional curvature 
on Wolf spaces as follows.

\begin{epr}\label{estimates}
Let $M$ be a Wolf space different from $\mathbb H P^m$ and let $P\subset \T M$ be any $2$-plane. Then the sectional curvature
$\kappa(P)$ satisfies the estimates
$$
   \frac{\s}{8n(n+2)}\cos^2(\theta_P)      \;  \le \;  \kappa(P) \;  \le \;      \frac{\s}{2n(n+2)} \ .
$$
In the case of the quaternionic projective space $\mathbb H P^m$ the sectional curvature $\kappa(P)$  satisfies
$$
\frac{\s}{16n(n+2)}     \;  \le \;  \kappa(P) \;  \le \;  \frac{\s}{4n(n+2)} \ .
$$
\end{epr}
\proof
Q.-S. Chi proved  in \cite[Thm. 1]{C} that the maximum 
of the sectional curvature on a positive quaternion-K\"ahler manifold is obtained on quaternionic planes. Moreover,  S. Helgason proved
in  \cite[Thm. 1.1]{H} that on each compact irreducible Riemannian symmetric space the maximum of the sectional
curvature is given by the length of the highest restricted root. This root coincides with the so-called Wolf root on all Wolf spaces except the
quaternion projective space. The length of the Wolf root was computed  in \cite[(4.1)]{SW}.  It follows that the maximum of the sectional 
curvature on compact Wolf spaces different from $\mathbb H P^n$ is $\frac{\s}{2n(n+2)}$. The length of all highest restricted roots can also
be found in \cite[Table 1, p. 175]{O}.

For the quaternionic projective space  $\mathbb H  P^n$ the estimates immediately follow from a well-known explicit curvature
formula for  $\mathbb H  P^n$. Indeed, for the  quaternionic projective space  $\mathbb H  P^n$  the sectional curvature of a 
plane $P$  spanned by an orthonormal basis $\{X,Y\}$ is given by 
$$
\kappa(P) \;=\; \kappa(X,Y) \;=\; \frac{\s}{16n(n+2)} \Big(1 + 3  \cos^2(\theta_P) \Big) \ .
$$
Thus the minimal value of the sectional curvature $\kappa(X, Y)$ is attained on totally real planes, whereas the maximal value is realised 
on quaternionic planes.
\qed


\bigskip

\labelsep .5cm

\end{document}